\pdfoutput=1
\documentclass{amsart}
\usepackage{
 amsmath,
 amssymb,
 enumerate,
 mathrsfs,  
 mathtools, 
 thmtools,
 tikz-cd,
 fullpage
}
\usepackage[pagebackref]{hyperref}

\usepackage{cleveref} 

\newcommand{\cB}{\mathcal{B}}
\newcommand{\cE}{\mathcal{E}}
\newcommand{\cF}{\mathcal{F}}
\newcommand{\cG}{\mathcal{G}}
\newcommand{\cH}{\mathcal{H}}
\newcommand{\cL}{\mathcal{L}}
\newcommand{\cO}{\mathcal{O}}
\newcommand{\cT}{\mathcal{T}}

\newcommand{\fa}{\mathfrak{a}}
\newcommand{\fb}{\mathfrak{b}}
\newcommand{\fS}{\mathfrak{S}}

\newcommand{\CC}{\mathbf{C}}
\newcommand{\FF}{\mathbf{F}}
\newcommand{\QQ}{\mathbf{Q}}
\newcommand{\RR}{\mathbf{R}}
\newcommand{\ZZ}{\mathbf{Z}}

\newcommand{\wG}{\widetilde{G}}
\newcommand{\wH}{\widetilde{H}}

\newcommand{\Af}{\mathbf{A}_\mathrm{f}} 
\newcommand{\Qp}{\QQ_p}
\newcommand{\Zp}{\ZZ_p}
\newcommand{\QQbar}{\overline{\QQ}}

\DeclareMathOperator{\br}{br}
\DeclareMathOperator{\GL}{GL}
\DeclareMathOperator{\GSp}{GSp}
\DeclareMathOperator{\mom}{mom}
\DeclareMathOperator{\pr}{pr}
\DeclareMathOperator{\Sym}{Sym}

\newcommand{\et}{\text{\textup{\'et}}}
\newcommand{\hw}{\mathrm{hw}}
\newcommand{\Iw}{\mathrm{Iw}}
\newcommand{\univ}{\mathrm{univ}}

\newcommand{\tps}[2]{\texorpdfstring{#1}{#2}}
\newcommand{\into}{\hookrightarrow}
\newcommand{\onto}{\twoheadrightarrow}

\theoremstyle{plain}
\newtheorem{theorem}{Theorem}[subsection]
\newtheorem{lemma}[theorem]{Lemma}
\newtheorem{proposition}[theorem]{Proposition}
\newtheorem{corollary}[theorem]{Corollary}
\newtheorem{definition}[theorem]{Definition}
\newtheorem{assumption}[theorem]{Assumption}

\theoremstyle{remark}
\declaretheorem[name=Remark,sibling=theorem,qed={\lower-0.3ex\hbox{$\diamond$}}]{remark}

\renewcommand{\le}{\leqslant}
\renewcommand{\ge}{\geqslant}

\newcommand{\eQ}{\mathop{e_Q'}}
\newcommand{\eQG}{\mathop{e_{Q_G}'}}

\title[Spherical varieties and $p$-adic families]
{Spherical varieties and $p$-adic families of cohomology classes}

\author{David Loeffler}
\address[D.~Loeffler]{Faculty of Mathematics and Computer Science, UniDistance Suisse, Schinerstrasse 18, 3900 Brig, Switzerland}
\email{david.loeffler@unidistance.ch}
\urladdr{\href{http://orcid.org/0000-0001-9069-1877}{0000-0001-9069-1877}}

\author{Robert Rockwood}
\address[R.~Rockwood]{Department of Mathematics, King's College London, Strand, London
WC2R 2LS, United Kingdom}
\email{robert.rockwood@kcl.ac.uk}

\author{Sarah Livia Zerbes}
\address[S.L.~Zerbes]{Department of Mathematics, ETH Z\"urich, R\"amistrasse 101, 8092 Z\"urich, Switzerland}
\email{sarah.zerbes@math.ethz.ch}
\urladdr{\href{http://orcid.org/0000-0001-8650-9622}{0000-0001-8650-9622}}

\thanks{We gratefully acknowledge support from the following research grants:
EPSRC Standard Grant EP/S020977/1 and ERC Consolidator Grant 101001051 ``ShimBSD'' (Loeffler);
Royal Society URF Enhancement Award RGF/EA/180212 (Rockwood).
}
\begin{document}

\begin{abstract}
 We prove a ``twist-compatibility'' result for $p$-adic families of cohomology classes associated to symmetric spaces. This shows that a single family of classes (lying in a finitely-generated Iwasawa module) interpolates classical cohomology classes of many different weights, including twists by Gr\"ossencharacters of possibly non-trivial infinity-type. This subsumes and generalises a number of prior results relating to Euler systems and $p$-adic $L$-functions, and we conclude with some novel applications to Euler systems for $\GSp_4$, $\GSp_4 \times \GL_2$, and $\GSp_4 \times \GL_2 \times \GL_2$.
\end{abstract}

\maketitle

\begin{center}
\emph{This paper is dedicated to Massimo Bertolini, to mark the occasion of his 60th birthday. Massimo's beautiful work has been an inspiration to us throughout our careers and we wish him all the best for the future.}
\end{center}

\section{Introduction}

 This article is a sequel to the article \cite{loeffler-spherical} of the first author. In the previous article, we constructed norm-compatible families of cohomology classes for Shimura varieties and other symmetric spaces, associated to pairs of reductive groups satisfying an ``open orbit'' condition. This very general construction subsumes a number of well-known results in Iwasawa theory, including the norm-compatibility of modular symbols and Heegner points for $\GL_2$, as well as more recent constructions such as the Euler system for $\GSp_4$ constructed in \cite{LSZ21}.

 In this paper, we consider varying the algebraic weight defining the coefficient system for our cohomology. Our main result, \cref{mainthm}, shows that the norm-compatible families constructed in \cite{loeffler-spherical} interpolate classical cohomology classes of many different weights (including twists by Gr\"ossencharacters of possibly non-trivial infinity-type). Several specific instances of this result are already known -- see \cref{rmk:knowncases} below -- and form an important ingredient in the proofs of ``explicit reciprocity laws'' for Euler systems and $p$-adic $L$-functions. The general theorem proved here subsumes all of these earlier results, and we conclude with several new applications of this theory to Euler systems for $\GSp_4$, $\GSp_4 \times \GL_2$, and $\GSp_4 \times \GL_2 \times \GL_2$.

\section{Control theorems for ordinary cohomology}

 The aim of this section is to prove a general control result (\cref{thm:rockwood}) for ``$Q$-ordinary'' cohomology groups associated to a general parabolic subgroup $Q$. Our formulation is heavily influenced by the work of Urban \cite{urban05} in the setting of Coleman theory; however, unlike Urban we obtain integral results, without inverting $p$ in the coefficients, as a consequence of imposing the stronger condition of $Q$-ordinarity rather than just $Q$-finite-slope. These results are adapted from the second author's 2022 University of Warwick PhD thesis.

 \subsection{Setup} \label{sect:setup}

  Let $p$ be a prime. We consider the following situation:

  \begin{itemize}
   \item $\cG$ is a connected reductive group over $\QQ$.

   \item We assume there exists an unramified $\Zp$-model $G$ of $\cG$, i.e.~a reductive group scheme $G$ over $\Zp$ together with an isomorphism $\cG \times_{\QQ} \Qp \cong G \times_{\Zp} \Qp$, and we fix a choice of such a model.

   \item $Q$ is a parabolic subgroup of $G$, and $\bar{Q}$ a choice of opposite parabolic; thus $L = Q \cap \bar{Q}$ is a Levi subgroup of $Q$, and the ``big cell'' $\overline{N} \times L \times N$ is an open subscheme of $G$ over $\Zp$, where $N$ is the unipotent radical of $Q$ (and similarly $\overline{N}$). Note that $Q$ and $\bar{Q}$ are local at $p$; we do not assume that they come from parabolic subgroups of $\cG$ defined over $\QQ$.

   \item $S$ denotes the torus $L / L^{\mathrm{der}}$, and $X^\bullet(S)$ its character lattice.
   We let $X_+^\bullet(S)$ be the cone of dominant weights (relative to $Q$).

   \item $A$ is the maximal $\Qp$-split subtorus in the centre of $L$.
  \end{itemize}

  We choose a torus $S^0 \subset S$, and let $L^0$ and $Q^0$ be its preimages in $L$ and $Q$ respectively (so $Q^0 = L^0 \ltimes N$). We write $\fS$ for the profinite abelian group $L(\Zp) / L^0(\Zp) \subseteq (S/S^0)(\Zp)$. The base ring for our constructions will be the Iwasawa algebra (completed group ring) $\cO[[\fS]]$, for $\cO$ a finite extension of $\Zp$. Note that $\fS$ is $p$-torsion-free (since $S / S^0$ splits over an unramified extension).

  We shall also make the following global assumption (Milne's ``Axiom (SV5)''):

  \begin{assumption}
   The torus $\mathcal{Z} = Z(\cG)^\circ$ (the identity compononent of the centre of $\cG$) is isogenous to the product of a $\QQ$-split torus and an $\RR$-isotropic torus. Equivalently, $\mathcal{Z}(\QQ)$ is discrete in $\mathcal{Z}(\Af)$.
  \end{assumption}

  \begin{remark}
   The assumption (SV5) can be weakened, but at the cost of more complicated statements: very roughly, if $\mathcal{Z}$ contains a torus which is $\RR$-split but not $\QQ$-split, we need to assume that this torus is contained in $S^0$. An important example where (SV5) does \textit{not} hold is the group $\cG = \operatorname{Res}_{F/\QQ} \GL_2$ for $F$ a totally real field; this case will be treated in detail in forthcoming work of Arshay Sheth (see \cite{sheth24a}).
  \end{remark}

 \subsection{Algebraic representations}

  Let us fix a maximal torus $T \subset L$, and a finite extension $K / \Qp$, with ring of integers $\cO$, such that $T$ splits over $K$. We can and do assume that $K$ is unramified (since $G$ has a smooth integral model).

  For $\lambda \in X^\bullet_+(S)$, there is a unique isomorphism class of irreducible representations $(\rho_\lambda, V_\lambda)$ of $G_{/K}$ of highest weight $\lambda$ (with respect to any choice of Borel contained in $Q$). 
  A representative of this isomorphism class can be constructed using the  Borel--Weil theorem, as the space of all polynomials
  \[ \{ f \in K[G] : f(\overline{n} \ell g) = \lambda(\ell) f(g)\quad \forall \overline{n}\in \overline{N}, \ell \in L, g \in G\}, \]
  with $G$ acting by right-translation. In particular, the highest weight space of $V_{\lambda}$ has a unique basis vector $f_{\lambda}^{\hw}$ whose restriction to the big Bruhat cell $\overline{N} L N$ is given by $\overline{n} \ell n \mapsto \lambda(\ell)$. This is also the highest \emph{relative} weight space of $V_\lambda$, i.e.~the highest-weight eigenspace for the action of $A$.

 \subsection{Integral lattices} Let $\lambda \in X^\bullet_+(S)$.

  \begin{definition}
   An \emph{admissible lattice} in $V_{\lambda}$ is an $\cO$-lattice $\cL \subset V_{\lambda}$ which is invariant under $G_{/\cO}$ (i.e.~$G(R)$ preserves $\cL \otimes R$ for every $\cO$-algebra $R$), and whose intersection with the highest-weight subspace is $\cO \cdot f_{\lambda}^{\hw}$.
  \end{definition}

  We recall the following standard facts, which can be found in \cite{lin92} for example:
  \begin{enumerate}[(a)]
   \item There are only finitely many admissible lattices in $V_{\lambda}$.
   \item Every admissible lattice is the direct sum of its intersections with the weight spaces of $V^G_\lambda$.
   \item There are unique maximal and minimal admissible lattices $\cL^{\max}$ and $\cL^{\min}$ such that every admissible $\cL$ satisfies $\cL^{\max} \supseteq \cL \supseteq \cL^{\min}$.
   \item If we realise $V_\lambda$ via the Borel--Weil construction as above, then the maximal admissible lattice $\cL^{\max}$ is given by the intersection of $V_{\lambda}$ with the coordinate ring of $G_{/\cO}$.
   \item The dual of an admissible lattice in $V_\lambda$ is an admissible lattice in $V_{\lambda}^\vee$. In particular, the dual of the maximal lattice is the minimal lattice and vice versa.
  \end{enumerate}

  We shall henceforth always write $V_{\lambda, \cO}$ for the \emph{maximal} admissible lattice, realised via the Borel--Weil construction as above, and $V^{\min}_{\lambda, \cO}$ for the minimal lattice.

  \begin{remark}
   Note that if $G = \GL_2$ and $V_\lambda = \Sym^n(K^2)$ is the $n$-th symmetric power of the standard representation, then $\Sym^n(\cO^2)$ is the maximal lattice, and the module $\operatorname{TSym}^n(\cO^2)$ of symmetric tensors (which features prominently in \cite{KLZ17}) is the minimal lattice.
  \end{remark}

  \begin{definition}
   Let $\Sigma^+$ denote the monoid of cocharacters $\eta \in X_\bullet(A)$ which are dominant with respect to $Q$, so that $\langle \eta, \phi\rangle \ge 0$ for every relative root $\phi$. We let $\Sigma^{++} \subset \Sigma^+$ be the ideal of elements that are strictly  dominant, so $\langle \eta, \phi \rangle > 0$ for all $\phi$.
  \end{definition}

  \begin{lemma}
   \label{lemma:preservelattice}
   If $\eta \in \Sigma^+$, then the endomorphism $p^{\langle \eta, \lambda\rangle} \rho_\lambda\left(\eta(p)^{-1}\right)$ of $V_\lambda$ acts on every weight space as a non-negative power of $p$, and hence restricts to an endomorphism of $\cL$ for any admissible lattice $\cL$. If $\eta \in \Sigma^{++}$ then this endomorphism acts as a strictly positive power of $p$ on every weight space except the highest one.
  \end{lemma}

  \begin{proof}
   This is an easy explicit check.
  \end{proof}

 \subsection{Cohomology of symmetric spaces}

  Let us fix a prime-to-$p$ level group $K^p$, which we shall suppose to be \emph{neat}.\footnote{See \cite[\S 0.1]{pink90} for a definition of this condition (which implies that the intersection of $K^p$ with any conjugate of $G(\QQ)$ is torsion-free) and a proof that any open compact subgroup contains a neat subgroup of finite index.} We define $H^*(U, \cO)$, for $U \subset G(\Qp)$ open compact, to mean \emph{one} of the following:
  \begin{itemize}
   \item Betti cohomology of the locally symmetric space $Y_G(K^p U)$, viewed as a real manifold, with coefficients in $\cO$.

   \item Supposing $\cG$ to be equipped with a choice of Shimura datum: the \'etale cohomology (with $\cO$ coefficients) of the canonical model of the Shimura variety of level $K^p U$, base-extended from the reflex field $E$ to $\QQbar$ (so the resulting module has a $\operatorname{Gal}(\QQbar/E)$-action).

   \item Again in the Shimura-variety setting, and supposing that the Shimura datum is of abelian type, and $\mathcal{Z}$ splits over a CM field: the \'etale cohomology of the canonical integral model of the Shimura variety over $\cO_E[1/S]$, where $E$ is the reflex field, and $S$ is a sufficiently large finite set of primes. In this case, we may also allow arbitrary cyclotomic twists of the coefficients.
  \end{itemize}

  One can also consider cohomology with compact support (in any of the above settings). More generally, since we may regard an algebraic representation of $G$ as a coefficient sheaf (for any of the above cohomology theories), we can make sense of $H^*(U, V_{\lambda, \cO})$ for any $\lambda$, as long as we assume $U \subseteq G(\Zp)$.

  \begin{remark}
   For the existence of canonical integral models of abelian-type Shimura varieties, see \cite[Theorem 2.2.1]{lovering17}; it suffices to take $S$ to contain $p$ and all places where $K^p$ is not hyperspecial. (In particular, $S$ can be chosen independently of $U$.)
  \end{remark}

  Note that in all three cases the groups $H^*(U, V_{\lambda, \cO})$ are finitely-generated over $\cO$. For Betti cohomology, and for \'etale cohomology over $\QQbar$, this is well-known. The case of \'etale cohomology of $\cO_E[1/S]$-models is slightly more delicate. We can express the \'etale cohomology of the integral model as the limit of a Hochschild--Serre spectral sequence whose $E_2$ page is $H^i(E^S / E, H^j_{\et}(Y_G(K^p U)_{\QQbar}, \dots))$, where $E^S$ is the maximal algebraic extension of $E$ unramified outside $S$. As the functors $H^i(E^S / E, -)$ map finitely-generated $\cO$-modules to finitely-generated $\cO$-modules, and vanish for $i \gg 0$, the result follows.

  \begin{remark}
   This would not work if we worked with the canonical model over $E$, rather than over $\cO_E[1/S]$, since the Galois cohomology functors $H^i(E, -)$ without ramification restrictions do not send finitely-generated modules to finitely-generated modules.
  \end{remark}

 \subsection{Hecke actions}

  For general $\lambda$, the functors $H^*(-, V_{\lambda, \cO})$ are \emph{Cartesian cohomology functors} for $G(\Zp)$, in the sense of \cite[Definition 2.1.3]{loeffler-spherical}; that is, given open compacts $U, V \subseteq G(\Zp)$ and $g \in G(\Zp)$ such that $g^{-1} U g \subseteq V$, we have pullback maps $[g]^* : H^*(V, V_{\lambda, \cO}) \to H^*(U, V_{\lambda, \cO})$, and pushforward maps $[g]_*$ in the opposite direction, satisfying an appropriate compatibility relation involving double cosets.

  However, if the coefficients are non-trivial, they are typically not cohomology functors for the whole of $G(\Qp)$ since this group does not act on the lattice $V_{\lambda, \cO}$. We shall work around this by renormalising our Hecke operators using \cref{lemma:preservelattice}.

  \begin{definition}\label{def:integralHecke}
   Suppose $U \subseteq G(\Zp)$. Then, for any $\lambda \in X^\bullet_+(S)$ and $\eta \in \Sigma^+$, we can define a normalised Hecke operator acting on $H^*(U, V_{\lambda, \cO})$ by
   \[ \cT'_\eta = p^h [U\eta(p^{-1}) U], \qquad h = {\langle \lambda, \eta \rangle}. \]
  \end{definition}

  \begin{remark}
   Recall $\eta$ is a cocharacter $\mathbf{G}_m \to A$, so $\eta(p^{-1}) \in A(\Qp) \subseteq G(\Qp)$. In the special cases of this theory we have previously worked out, the scalar factor $p^h$ appearing above is hidden from view, since we had always chosen $\eta$ and $\lambda$ such that $\langle \eta, \lambda\rangle = 0$.
  \end{remark}

  \begin{definition}
   We say an open compact subgroup $U \subset G(\Zp)$ has an \emph{Iwahori decomposition} (with respect to $Q$) if we have $U = \bar{N}_U L_U N_U$, where $N_U = U \cap N(\Zp)$ etc, and $N_U$ and $\bar{N}_U$ are the products of their intersections with the relative root spaces for $A$.
  \end{definition}

  The prototypical examples are the parahoric subgroups $J = \{ g \in G(\Zp): g \bmod p \in Q(\FF_p)\}$ and its opposite $\bar{J}$.

  \begin{proposition}
   \label{prop:idempotent}
   If $U$ has an Iwahori decomposition, then the operators $\cT'_\eta$ for varying $\eta$ define an action of $\Sigma^+$ on $H^*(U, V_{\lambda, \cO})$; and if we choose any $\eta_0 \in \Sigma^{++}$, the \emph{anti-ordinary idempotent}
   \[ \eQ = \lim_{k \to \infty} (\cT'_{\eta_0})^k \]
   is independent of the choice of $\eta_0$, and cuts out the maximal direct summand on which the operators $\cT_\eta'$ are all invertible.
  \end{proposition}

  \begin{remark}
   We will show in \cref{rmk:indepU} below that $\eQ H^*(U, V_{\lambda, \cO})$ in fact depends only on $U \cap Q(\Zp)$.

   We use the term ``anti-ordinary'' since our Hecke operators are defined using the image of $p^{-1}$ under a dominant cocharacter, rather than $p$. In the familiar case when $G = \GL_2$ and $Q$ is the upper-triangular Borel, this means we do not work with the Hecke operator $U_p$ which acts as $\sum a_n q^n \mapsto \sum a_{np} q^n$ on $q$-expansions, but rather its transpose $U_p'$.
  \end{remark}

  \begin{proposition}\label{prop:comparelattices}
   For any $U$ with an Iwahori decomposition, the natural map
   \[
    \eQ H^i(U, V_{\lambda, \cO}^{\min}) \to \eQ H^i(U, V_{\lambda, \cO}),
   \]
   induced by the inclusion $V_{\lambda, \cO}^{\min} \into V_{\lambda, \cO}$, is an isomorphism.
  \end{proposition}

  \begin{proof}
   Both $V_{\lambda, \cO}^{\min}$ and $V_{\lambda, \cO}$ are direct sums of their weight spaces for $A$. The inclusion map is an isomorphism on the highest relative weight space (by definition). On each lower weight space, $p^{\langle \lambda, \eta_0\rangle} \rho_\lambda(\eta_0(p)^{-1})$ acts as a positive power of $p$, so some power of this map will annihilate the quotient of the weight spaces in the two lattices. Thus, if $r \gg 0$ (depending on $\lambda$), the operator $(\cT'_{\eta_0})^r$ annihilates the cohomology of the quotient.
  \end{proof}

 \subsection{Iwasawa cohomology and control theorems}

  \begin{definition}
   \label{def:iwacoh}
   For $i \ge 0$, and $\lambda \in X^\bullet_+(S)$, we set
   \[ H^i_{\Iw}(Q^0, V_{\lambda, \cO}) = \varprojlim_U H^i(U, V_{\lambda, \cO}), \]
   where $U$ varies over open compact subgroups of $G(\Zp)$ containing $Q^0(\Zp)$, and the inverse limit is taken relative to the pushforward maps.
  \end{definition}

  One checks that the levels $U$ having an Iwahori decomposition are cofinal in the above inverse system, and the operators $\{ \cT'_\eta : \eta \in \Sigma^+\}$ are compatible with the pushforward maps and hence act on the inverse limit.

  In this section and the next, we prove the following result, generalising results due to Ohta \cite{ohta99, ohta00} for Betti or \'etale cohomology of modular curves to arbitrary reductive groups:

  \begin{theorem} \
  \label{thm:rockwood}
   \begin{enumerate}[(i)]
    \item The modules $\eQ H^i_{\Iw}(Q^0, \cO)$ are finitely-generated over $\cO[[\fS]]$, and vanish for $i \gg 0$.


    \item Let $U$ be any open compact subgroup of $G(\Zp)$ such that $U \supseteq Q^0(\Zp)$ and $U$ admits an Iwahori decomposition. For any $\lambda \in X^\bullet_+(S / S_0)$, there is a convergent spectral sequence (supported in the second quadrant $i \le 0, j \ge 0$):
    \[ E_2^{ij} = \operatorname{Tor}_{-i}^{\cO[[\fS_U]]}\Big( \eQ H^j_{\Iw}(Q^0, \cO), \cO[-\lambda] \Big) \Rightarrow \eQ H^{i+j}(U, V_{\lambda, \cO}),\]
    where $\fS_U$ is the image of $U \cap Q(\Zp)$ in $\fS$, and $\cO[-\lambda]$ denotes $\cO$ with $\fS_U$ acting via the algebraic character $-\lambda$. This spectral sequence is compatible with the action of Hecke operators away from $p$, and with the operators $\{ \cT'_\eta: \eta \in \Sigma^+\}$.
   \end{enumerate}
  \end{theorem}

  We shall give the proof below assuming that $H^*(-)$ denotes Betti cohomology. The remaining cases follow readily from this, as follows. One may define all the spaces and maps for a finite truncation of the coefficients, using general functoriality properties of \'etale cohomology. Then the remaining assertions -- that the inverse limits are finitely-generated, and that the moment map in (ii) is an isomorphism on the ordinary part -- can be checked using Artin's comparison theorem to relate Betti cohomology and \'etale cohomology over $\QQbar$, and the Hochschild--Serre spectral sequence in order to descend to the cohomology of the integral model over $\cO_E[1/S]$.

  \begin{remark}\label{rmk:coset}
   In the above formulation, we have assumed that $\lambda$ factors through the quotient $S / S_0$, and in particular is trivial on $L^{\mathrm{der}}$. More generally, we can fix any $\nu \in X^\bullet_+(T)$ (not necessarily trivial on $L^{\mathrm{der}}$) and study the cohomology groups $\eQ \cdot \varprojlim_U H^i(U, V_{\nu, \cO})$; one can show that these are finite-type $\cO[[\fS]]$-modules which interpolate the cohomology groups $\eQ \cdot H^i(U, V_{\lambda + \nu, \cO})$ as $\lambda$ varies over characters of $S / S_0$ (so $\lambda + \nu$ varies over a coset of $X^\bullet(S / S_0)$ in $X^\bullet(T)$). We shall stick to the case $\nu = 0$ here for ease of notation, but the proof in the general case is very similar.
  \end{remark}

 \subsection{Proof of the control theorem}

  For brevity we set $R \coloneqq \cO[[\fS]]$. Let us choose an element $\eta \in \Sigma^{++}$ and set $\tau = \eta(p)$.

  \subsubsection*{Groups} Let $Q^0(\Zp) \subset U \subset G(\Zp)$ be an open compact subgroup admitting an Iwahori decomposition
  \[
   U = \bar{N}_U \times L_U \times N_0,
  \]
  and for $r \ge 0$ set
  \[
   U^{(r)} = \tau^{-r}U\tau^r \cap U.
  \]
  Note that we have $U^{(r)} = \bar{N}_U^{(r)} \times L_U \times N_0$ for all $r$, where $\bar{N}_U^{(r)} = \tau^{-r} \bar{N}_U \tau^r$; in particular $\bigcap_{r \ge 0} U^{(r)} = U \cap Q(\Zp)$.

  \subsubsection*{Betti cohomology complexes}

   For an open compact subgroup $K = K^pU \subset G(\Af)$ the symmetric space $Y(K)$ is given by a finite disjoint union
   \[
     Y(K) = \bigsqcup_i \mathcal{X}/\Gamma_i.
   \]
   where $\Gamma_i \subset G(\QQ)^+$ are arithmetic subgroups, and $\mathcal{X} = G(\RR)^+/\left(K_{\infty}Z(\RR) \cap G^+(\RR)\right)$ with $K_{\infty}$ a maximal compact subgroup of $G(\RR)$. If $K$ is neat then $\mathcal{X}/\Gamma_i$ has the structure of a smooth manifold. This space satisfies $\pi_1\left(\mathcal{X}/\Gamma_i\right) = \Gamma_i$ and thus
   \[
     H^j(K, V_{\lambda}) = \bigoplus_iH^j(\Gamma_i, V_{\lambda}).
   \]
   Let $M$ be a $\ZZ[K]$-module with $K$ acting through its projection to $U$. By the theory of Borel--Serre (\cite{borelserre73}; see also \cite[\S 4.2.1]{urban05}), the cohomology groups $H^i(\Gamma_i, M)$ vanish for $i \gg 0$, and are computed by a complex $\mathcal{C}^{\bullet}(\Gamma_i, M)$ of $\ZZ[\Gamma_i]$-modules, depending functorially on $M$, each term of which is isomorphic to a finite direct sum of copies of $M$. We write
   \[
      \mathcal{C}^{\bullet}(U, M) := \bigoplus_i \mathcal{C}^{\bullet}(\Gamma_i, M).
   \]
   This complex computes the Betti cohomology of $Y(K)$. If $M$ has an action of $\tau^{-1}$ compatible with the $K$-action then we can define an action of $\cT'_{\eta}$ on $ \mathcal{C}^{\bullet}(K, M)$ lifting the action on cohomology. There are many such liftings but these differ by a chain homotopy and thus induce the same morphism in the derived category.

  \subsubsection*{Modules of continuous functions}

   For $\lambda \in X_+^{\bullet}(S / S_0)$ we define $U$-modules
   \begin{align*}
    C_{\lambda, U} &= \{f: U \to \cO : f \ \text{continuous}, \ f(n \ell g) = \lambda^{-1}(\ell)f(g)  \ \forall n \in N_0, \ell \in L_U, g \in U\} \\
    C_{\univ} &= \{f: U \to \cO: f \ \text{continuous}, \ f(n \ell g) = f(g) \ \forall n \ell \in Q^0(\Zp), g \in U\}
   \end{align*}
   with $U$ acting by right translation.

   Note that, since we have $U = N_0 \times L_U \times \bar{N}_U$, restriction to $\bar{N}_U$ is an isomorphism between $C_{\lambda, U}$ and the space of continuous functions on $\bar{N}_U$, for any $\lambda$ (but the action of $U$ on this space depends on $\lambda$). Similarly, functions in $C_{\univ}$ are uniquely determined by their restriction to $L_U \times \bar{N}_U$, and this restriction must transform trivially under left-translation by $L_U \cap L^0$, giving an isomorphism between $C_{\univ}$ and continuous functions on $\fS_U \times \bar{N}_U$.

   We endow these spaces with an action of $\tau$ given by
   \[
    \left(\tau\cdot f\right)(n \ell \bar{n}) = f(n\ell \tau^{-1}\bar{n}\tau),
   \]
   and with the uniform-convergence topology. There are natural inclusions
   \[
    V_{\lambda^{\vee}, \cO}^{\mathrm{max}} \to C_{\lambda, U} \to C_{\univ}
   \]
   intertwining the $U$ and $\tau$ actions, where $\lambda^{\vee} = -w_0\lambda$ and $w_0$ is the long Weyl element.

\subsubsection*{Modules of measures} We define modules of bounded distributions
\begin{align*}
    D_{\lambda, U} &= \mathrm{Hom}_{\mathrm{cont}}\left(C_{\lambda, U}, \cO \right) \\
    D_{\univ} &= \mathrm{Hom}_{\mathrm{cont}}\left(C_{\univ}, \cO \right),
\end{align*}
which inherit actions of $U$ and $\tau^{-1}$ by duality.
By dualising the above sequence, these spaces admit the following maps
$$
    D_{\univ} \to D_{\lambda, U} \to V_{\lambda, \cO}^{\min},
$$
which are equivariant for the actions of $U$ and $\tau^{-1}$.
Writing $\fS_U$ for the image of $U$ in $\fS$, these groups are naturally modules over $R = \cO[[\fS_{U}]]$ and this structure is given explicitly by the isomorphisms:
\begin{align*}
    D_{\lambda, U} &\cong (\cO[-\lambda] \otimes_{\cO} R)\hat{\otimes}_{\cO}\cO[[ \bar{N}_U]], \\
    D_{\univ} &\cong R \hat{\otimes}_{\cO} \cO[[\bar{N}_{U}]],
\end{align*}
where $\cO[-\lambda]$ denotes $\cO$ regarded as an $\fS$ (and accordingly $\cO[[\fS]]$) module via the inverse of $\lambda$.
In particular
\[ D_{\univ} \otimes_R \cO[-\lambda] \cong D_{\lambda, U} \]
as $R$-modules. (In this case the completed tensor product and the abstract tensor product agree, since both modules are complete and $\cO[-\lambda]$ is finitely-generated.)

\begin{lemma}
        The $R$-module $D_{\univ}$ is flat.
    \end{lemma}
    \begin{proof}
    We have shown that $D_{\univ}$ is isomorphic as an $R$-module to a power series ring in finitely many variables over $R$, which is in particular flat
    \end{proof}

     \begin{proposition} \label{prop:idem}
     There is an idempotent $\eQ$ acting on $\mathcal{C}^{\bullet}(U, D_{\univ})$ such that
     \begin{itemize}
         \item For all $i$ and all $f \in \mathcal{C}^i(U, D_{\univ})$ we have $\eQ f = \lim_{n \to \infty}\left(\cT'_{\eta}\right)^{n!}f$.
         \item There is a decomposition
         $$
            \mathcal{C}^{\bullet}(U, D_{\univ}) = \eQ\mathcal{C}^{\bullet}(U, D_{\univ}) \oplus \left( 1 - \eQ\right)\mathcal{C}^{\bullet}(U, D_{\univ})
         $$
         such that $\cT'_{\eta}$ acts invertibly on the first component and is topologically nilpotent on the second.
     \end{itemize}
     Moreover, the terms of the complex $\eQ\mathcal{C}(U, D_{\univ})$ are flat $R$-modules.
 \end{proposition}

 \begin{proof}
  Note that $D_{\univ}$ is a profinite $R$-module, isomorphic to a power series ring in finitely many variables over $R$; and we can write it as an inverse limit of finite quotients, $D_\univ = \varprojlim_r D_{\univ} / W_r$, where each $W_r$ is stable under the action of $U$ and of $\tau^{-1}$. So $\mathcal{C}^{\bullet}(U, D_{\univ})$ is a complex of profinite $R$-modules, and the $W_r$ induce subcomplexes $W^{\bullet}_r$ whose components are bases of of open neighbourhoods for the components of $\mathcal{C}^{\bullet}(U, D_{\univ})$. The action of $\cT'_{\eta}$ on $\mathcal{C}^{\bullet}(U, D_{\univ})$ preserves $W^{\bullet}_r$ and thus descends to an action on $\mathcal{C}^{\bullet}(U, D_{\univ})/W_r^{\bullet}$ for every $r$. This complex consists of modules of finite cardinality. We can thus use the ideas developed in \cite[Section 2.1]{pilloni20} to deduce the existence of $\eQ$.

  Since each component of $\mathcal{C}^{\bullet}(U, D_{\univ})$ is a direct sum of copies of $D_{\univ}$ (by the Borel--Serre theory mentioned above), this is a complex of flat $R$-modules by the previous lemma. Since $\eQ\mathcal{C}^{\bullet}(U, D_{\univ})$ is a direct summand of $\mathcal{C}^{\bullet}(U, D_{\univ})$ we conclude that it is also a complex of flat $R$-modules.
 \end{proof}

 \begin{remark}
  Note that the idempotent $e_Q'$ is independent of the choice of $\eta$, since the multiples of any given $\eta$ are cofinal in $\Sigma^{++}$.
 \end{remark}

The key input into the proof of the control theorem is the following lemma.

\begin{lemma} \label{lem:hida}
Let $M$ be a  $\ZZ[U]$-module with a compatible action of $\tau^{-1}$. The following diagram commutes on cohomology
$$
\begin{tikzcd}
\mathcal{C}^{\bullet}(U^{(s)}, M) \arrow[d, "(\cT'_{\eta})^s"] \arrow[r, "\mathrm{cores}"] & \mathcal{C}^{\bullet}(U, M) \arrow[dl, "{[\tau^s]}_*"] \arrow[d, "(\cT'_{\eta})^s"] \\
\mathcal{C}^{\bullet}(U^{(s)}, M) \arrow[r, "\mathrm{cores}"'] & \mathcal{C}^{\bullet}(U, M)
\end{tikzcd}
$$
\end{lemma}
\begin{proof}
This is a standard argument in Hida theory, compare e.g.~\cite[Proposition 4.1]{hida95} or \cite[Lemma 3.1]{tilouineurban99}.
\end{proof}

  \begin{corollary}
   When $\mathcal{C}^{\bullet}(U, M)$ in addition admits an ordinary idempotent $e$ in the sense of Proposition \ref{prop:idem}, the corestriction maps induce isomorphisms
   \[
    eH^i(U^{(s)}, M) \cong eH^i(U, M),
   \]
   for any $s \ge 0$.
  \end{corollary}

  \begin{proof}
   On the image of the ordinary idempotent, the operator $(\cT'_{\eta})^s$ is invertible; so it follows from the previous proposition that the corestriction map is an isomorphism on this summand.
  \end{proof}

  \begin{remark}\label{rmk:indepU}
   In particular, if $U, U'$ are any open compacts in $G(\Zp)$ which both have Iwahori decompositions and satisfy $U \cap Q(\Zp) = U' \cap Q(\Zp)$, then $U^{(s)}$ is contained in $U \cap U'$ for all $s \gg 0$, and hence the pushforward map $\eQ H^*(U \cap U', M) \to \eQ H^*(U, M)$ is an isomorphism. The same holds with $U$ and $U'$ interchanged; so we deduce that $\eQ H^*(U, M)$ depends only on $U \cap Q(\Zp)$, as claimed above.
  \end{remark}

\begin{proposition}
There is an isomorphism
$$
    \eQ H^i(U, D_{\lambda, U}) \cong  \eQ H^i(U, V_{\lambda, \cO}).
$$
\end{proposition}
\begin{proof}
    We sketch the proof. From the previous corollary, for all $s \ge 0$ we have
    \[ \eQ H^i(U, V_{\lambda, \cO}) = \eQ H^i(U^{(s)}, V_{\lambda, \cO}) = \eQ H^i(U, \cO[U] \otimes_{\cO[U^{(s)}]} V_{\lambda, \cO}). \]
    (Here we are using the Cartesian property to see that the fibres of $Y_G(U^{(s)})$ over $Y_G(U)$ are naturally identified with $U / U^{(s)}$; this step would break down in the absence of condition (SV5).)

    From the Iwahori decomposition, we have $U / U^{(s)} \cong \bar{N}_U / \bar{N}_U^{(s)}$; and as $\bigcap_s \bar{N}_U^{(s)} = \{1\}$, the inverse limit of the profinite $U$-modules $\cO[U] \otimes_{\cO[U^{(s)}]} V_{\lambda, \cO}$ is identified with $\cO[[U]] \otimes_{\cO[[L_U N_0]]} V_{\lambda, \cO}$. Moreover, the action of $e_{Q}'$ kills all the $S$-eigenspaces in $V_{\lambda, \cO}$ except the highest-weight space, which is one-dimensional with $L_U$ acting via $\lambda$. So we obtain the cohomology of $\cO[[U]] \otimes_{\cO[[L_U N_0]]} \cO[\lambda] = D_{\lambda, U}$.
\end{proof}

\begin{corollary}
 \label{cor:eQiso}
    We have an isomorphism
$$
     \eQ H^i(U, D_{\univ}) \cong \eQ H^i_{\Iw}(Q^0, \cO).
$$
\end{corollary}
\begin{proof}
 This follows in a similar fashion to the above, identifying $D_{\univ}$ with the module of measures on the quotient $U / Q^0(\Zp)$.
\end{proof}

 Define $M^{\bullet}$ to be the image of  $ \eQ\mathcal{C}^{\bullet}(U, D_{\univ})$ in the subcategory $\mathcal{D}^{\mathrm{flat}}(R)$ of the derived category of $R$-modules generated by flat objects. Note that this is independent of the choice of $U$, since the maps induced by shrinking $U$ are quasi-isomorphisms from the last proposition.

\begin{proposition}
    Suppose $\eQ\mathcal{C}^{\bullet}(U, D_{\univ})$ is concentrated in degrees $[a,b]$. Then $M^{\bullet}$ is represented by a perfect complex concentrated in degrees $[a,b]$.
\end{proposition}
\begin{proof}
Let $U_1 = \{ u \in U : u \bmod p \in Q^0(\FF_p)\}$. Then the cohomology groups
$$
    H^i(M^{\bullet}\otimes_R R/J(R)) = \eQ H^i(U_1, \cO/p)
$$
are finite, and thus we can use the same method as \cite[Proposition 2.2.1]{pilloni20} to obtain the result. (The lemma as stated is for complete Noetherian \emph{local} rings, and our ring $R$ is only semi-local; but $R$ is isomorphic to the direct product of its localisations at the finitely many maximal ideals of $R$, so we may obtain the result for complexes of $R$-modules by applying the proposition to each localisation and taking the direct product.)
\end{proof}

 \begin{theorem}
For all $r \ge 1$ and $\lambda \in X_+^{\bullet}(S)$ there is a quasi-isomorphism
$$  M^{\bullet}\otimes^{\mathbb{L}}_{\cO[[\fS_U]]} \cO[-\lambda] \cong \eQ\mathcal{C}^{\bullet}(U, V_{\lambda})
$$
 \end{theorem}
 \begin{proof}
  Since $M^{\bullet}$ is represented by the flat complex $\eQ\mathcal{C}^{\bullet}(U, D_{\univ})$, we can compute the derived tensor product as
  \[
   \eQ\mathcal{C}^{\bullet}(U, D_{\univ})\otimes_{\cO[[\fS_U]]}\cO[-\lambda] = \eQ\mathcal{C}^{\bullet}(U, D_{\univ}\otimes_{\cO[[\fS_U]]}\cO[-\lambda])
  \]
  (where we do not need to take derived tensor product, as the terms of the complex are flat). By the previous lemmas this complex is quasi-isomorphic to $\eQ\mathcal{C}^{\bullet}(U, V_{\lambda})$.
 \end{proof}

 \begin{corollary}
     There is a spectral sequence
\begin{equation} \label{eq:spseq}
    E_2^{i, j}: \mathrm{Tor}^{\cO[[\fS_U]]}_{-i}(\eQ H^j_{\Iw}(Q^0, \cO), \cO[-\lambda]) \implies \eQ H^{i + j}(U, V_{\lambda})
\end{equation}
supported in the second quadrant.
 \end{corollary}

 \begin{proof} This follows from the previous theorem using the K\"unneth spectral sequence for the Tor functor (see e.g.~\cite[Theorem 5.6.4]{weibel94}).
 \end{proof}

\begin{definition}
\label{def:mommapfinite}
We denote by
\[ \mom^{\lambda}_U: \eQ H^j_{\Iw}(Q^0, \cO) \otimes_{\cO[[\fS_U]]} \cO[-\lambda] \longrightarrow \eQ H^i(U, V_{\lambda, \cO})\]
the edge map of the spectral sequence \eqref{eq:spseq}.
\end{definition}

We will need below an explicit description of this map as an inverse limit. The natural edge map
 $$
    \eQ H^i(U, D_{\univ}) \otimes_{\cO[[\fS_U]]}\cO[-\lambda] \to \eQ H^i(U, V_{\lambda, \cO})
 $$
 comes from the map
$$
    H^i(U, D_{\univ}) \to H^i(U, D_{\univ}\otimes V_{\lambda, \cO})
$$
induced by the map on coefficients given by $\mu \mapsto \mu \otimes f_{\lambda}^{\hw}$, composed with projection to level $U$. Since modulo $p^s$ we have $f^{\hw}_{\lambda}\in H^0(U^{(s)}, V_{\lambda, \cO}/(p^s))$ it's easy to see that the above map is given by the composite of
\begin{align*}
 \mathrm{mom}^{\lambda}:  H^i_{\Iw}(Q^0, \cO) \to H^i_{\Iw}(Q^0, V_{\lambda, \cO}) \\
  x \mapsto \varprojlim_s \left( x \cup f_{\lambda}^{\hw} \ \mathrm{mod} \ p^s\right ).
\end{align*}
and projection to level $U$.

  \section{Branching laws for algebraic representations}

 We recall the setup considered in \cite{loeffler-spherical}, where we consider pushforward of ``Eisenstein-type'' cohomology classes from a group $H$ to a larger group $G$. We shall fix an embedding $\cH \into \cG$ of reductive groups over $\QQ$, extending to an embedding $H \into G$ of reductive group schemes over $\Zp$.

 \subsection{Flag varieties}
  \label{sect:pfsetup}

  We fix choices of subgroups $Q, S$ etc as in \S \ref{sect:setup}, both for $G$ and for $H$; and we distinguish between them using subscripts, so $Q_G$ and $Q_H$ are parabolics in $G$ and $H$ respectively. Note that we do not assume any direct compatibility between these (e.g.~we do not suppose that $Q_H = Q_G \cap H$ etc).
  We are interested in the action of $H$, and its subgroup $Q_H^0$, on the flag variety $\cF = G / \overline{Q}_G$. We assume that there exists $u \in \cF(\Zp)$ such that
  \begin{enumerate}[(A)]
   \item the $Q_H^0$-orbit of $u$ is Zariski-open in $\cF$, and
   \item the image of $\overline{Q}_G \cap u^{-1}Q_H^0u$ under the projection $\overline{Q}_G \onto L_G \onto S_G$ is contained in the subtorus $S_G^0$.
  \end{enumerate}

  \begin{remark}
   Note that the validity of (B) is independent of the choice of representative $u$ in (A), since $\overline{Q}_G \cap u^{-1}Q_H^0u$ is well-defined up to conjugation (and $S_G$ is commutative).
  \end{remark}

 \subsection{Highest-weight representations}

  \begin{proposition}
   For any $\lambda \in X^\bullet_+(S_G / S_G^0)$, we have $\dim (V^G_\lambda)^{Q_H^0} \le 1$. If this dimension is 1, then $(V_\lambda^G)^{Q_H^0}$ is spanned by the unique polynomial $\br_{\lambda} \in K[G]$ (a ``branching polynomial'') such that $\br_\lambda\left(\overline{q} u^{-1} q'\right) = \lambda(\overline{q})$ for $\overline{q} \in \overline{Q}_G$ and $q' \in Q_H^0$.
  \end{proposition}

  \begin{proof}
   Suppose $f \in (V_\lambda^G)^{Q_H^0}$. Then the restriction of $f$ to the subvariety $\overline{Q}_G  u^{-1} Q_H^0$ of $G$ is uniquely determined by $f(u^{-1})$. Since $\overline{Q}_G  u^{-1} Q_H^0$ is open (and $G$ is connected), this implies that $f$ itself is uniquely determined, as an element of $K[G]$, by $f(u^{-1})$. In particular, if $f$ is non-zero, then we can scale it so that $f(u^{-1}) = 1$.
  \end{proof}

  \begin{remark}
   Conversely, if $S_G^0$ is exactly the image of $\overline{Q}_G \cap u^{-1}Q_H^0u$ in $S_G$, then any weight $\lambda \in X^\bullet_+(S_G)$ such that $(V^G_\lambda)^{Q_H^0} \ne 0$ must lie in $X^\bullet_+(S_G / S_G^0)$. However, it is convenient not to require $S_G^0$ to have this minimality property.
  \end{remark}

  \begin{definition}
   We shall say that a weight $\lambda \in X^\bullet(S_G / S_G^0)$ is \emph{$Q_H^0$-admissible} if $(V_\lambda^G)^{Q_H^0} \ne 0$.
  \end{definition}

  If this holds, then $Q_H / Q_H^0$ must act on $(V_\lambda^G)^{Q_H^0}$ via a weight $\mu \in X^\bullet(S_H / S^0_H)$, and this weight must be dominant (since it is the highest weight of the  $H$-representation spanned by $(V_\lambda^G)^{Q_H^0}$, which is finite-dimensional). We denote this weight by $\mu(\lambda)$, or just $\mu$ if $\lambda$ is clear from context.

  \begin{remark}
   The natural map
   \[ \frac{Q_H \cap u \overline{Q}_G u^{-1}}{Q_H^0 \cap u \overline{Q}_G u^{-1}} \to Q_H / Q_H^0 \]
   is injective, and for dimension reasons, it must in fact be an isomorphism. The left-hand side maps naturally to $S_G / S_G^0$; so we can characterise $\mu(\lambda)$ as the pullback of $\lambda$ to this quotient.
  \end{remark}

  \begin{proposition}
   \label{prop:branchingK}
   If $\lambda$ is $Q_H^0$-admissible, then there exists a unique homomorphism
   \[ V^H_{\mu} \to \left(V^G_{\lambda}\middle) \right|_H, \]
   where $\mu \in X^\bullet(L_H / L^0_H)$ is the character by which $Q_H$ acts on $\left(V^G_{\lambda}\right)^{Q_H^0}$. This is characterised by mapping the highest-weight vector $f^{\hw}_{\mu}$ of $V^H_{\mu}$ to $\br_{\lambda}$.
  \end{proposition}

  \begin{proof}
   This follows from Frobenius reciprocity: $V^H_{\mu}$ is isomorphic to the co-induced module $\operatorname{Coind}_{Q_H}^H(\mu)$ (which is left adjoint to restriction of representations, while the usual induction is right adjoint).
  \end{proof}

  \begin{proposition}
   If $\lambda$ is a $Q_H^0$-admissible weight, then $\br_{\lambda} \in V^G_{\lambda, \cO}$.
  \end{proposition}

  \begin{proof}
   Since $G_{/\cO}$ is smooth over $\cO$, it is a normal scheme, so it suffices to check that $\br_{\lambda}$ is regular away from a subvariety of codimension $\ge 2$ (Hartogs' lemma). By definition, it is regular on the generic fibre, and also on an open dense subvariety of $G_{/k}$, where $k = \cO/p\cO$; so we are done.
  \end{proof}

  \begin{corollary}
   \label{cor:branchingO}
   The branching map of \cref{prop:branchingK} restricts to a map of $\cO$-modules
   \[ V^{H, \min}_{\mu, \cO} \to V^G_{\lambda, \cO}, \]
   where $V^{H, \min}_{\mu, \cO}$ is the minimal admissible lattice in $V^H_\mu$.
  \end{corollary}

  \begin{remark}
   Compare \cite[Proposition 4.3.5]{LSZ21}. (We do not know whether the minimal lattice maps to the minimal lattice, nor whether the maximal one maps to the maximal one.)
  \end{remark}

 \subsection{Projection to the highest-weight subspace}

  If $\lambda$ is an $Q_H^0$-admissible weight, we thus have two canonical vectors in $V^G_{\lambda}$: namely, the highest-weight vector $f_\lambda^{\hw}$, and the $Q_H^0$-invariant vector $\br_{\lambda}$. We also have a canonical linear functional $\psi_\lambda$ on $V_\lambda$, given by evaluation at $\mathrm{id}_G$, which factors through projection from $V_\lambda$ to its highest-weight space (i.e.~it is a lowest-weight vector of weight $-\lambda$ for the contragredient representation $(V_\lambda^G)^\vee$).

  \begin{proposition}
  \label{trivprop}
   We have $\psi_\lambda(f_\lambda^{\hw}) = \psi_\lambda(u^{-1} \cdot \br_\lambda)$.
  \end{proposition}

  \begin{proof}
   Both sides are, by definition, equal to 1.
  \end{proof}

  \begin{corollary}\label{cor:integralrelation}
   Let $\eta \in \Sigma^{++}$. If $\lambda$ is $Q_H^0$-admissible, then for any $r \ge 1$ we have
    \[
     p^{hr} \rho_{\lambda}(\eta(p)^{-r}) \cdot (f_\lambda^{\hw} - u^{-1} \cdot \br_\lambda) \in p^r V_{\lambda, \cO},
    \]
   where $h = \langle \lambda, \eta\rangle$ as above.
  \end{corollary}

  \begin{proof}
   Since $V_{\lambda, \cO}$ is stable under the action of $G$, it is certainly stable under the action of $ \mathbf{G}_{m, /\Zp}$ defined by $\eta$, and hence is equal to the direct sum of its eigenspaces for the powers of $\eta$. So to prove that $p^h \rho_{\lambda}(\eta(p)^{-1})$ maps $V_{\lambda, \cO}$ to itself, it suffices to show that the eigenvalues of $\rho_{\lambda}(\eta(p)^{-1})$ on $V_\lambda$ are in $p^{-h} \Zp$. Clearly, these eigenvalues are given by $p^{-\langle \eta, \mu\rangle}$ where $\mu$ varies over the weights $V_\lambda$. However, since $\lambda$ is the highest weight of $V_\lambda$, and $\eta$ is dominant, all such weights satisfy $\langle \eta, \mu \rangle \le \langle \eta, \lambda\rangle = h$.

   To prove the second statement, it suffices to take $r = 1$. By \cref{trivprop}, $f_\lambda^{\hw} - u^{-1} \cdot \br_\lambda$ has zero projection to the highest relative weight space. So it is a sum of vectors lying in relative weight spaces of strictly smaller weights, and hence $p^h \rho_{\lambda}(\eta(p)^{-1})$ acts on these with eigenvalues in $p\Zp$.
  \end{proof}

\section{Norm-compatible families}

 \subsection{Subgroups}

  Let us choose $\eta \in \Sigma^{++}$ (which will remain fixed throughout the construction), and set $\tau = \eta(p)$. For $r \ge 0$, let us write
  \[ N_r = \tau^r N_G(\Zp) \tau^{-r},\qquad \overline{N}_r = \tau^{-r} \overline{N}_G(\Zp) \tau^r.\]
  Note that $N_r$ and $\overline{N}_r$ are contained (usually strictly) in the kernel of reduction modulo $p^r$. We set $L_r = \{ \ell \in L_G(\Zp): \ell \bmod{p^r} \in L_G^0\}$. Then the ``norm-compatibility machine'' of \cite{loeffler-spherical} uses the open compact subgroups $U_r$ and $V_r$ of $G(\Zp)$ (see \S 4.4 of \emph{op.cit.}) defined for $r \ge 1$ by
  \[ U_r = \overline{N}_0 L_r N_r,
  \qquad V_r = \overline{N}_r L_r N_0.\]
  Note that $(U_r)_{r \ge 1}$ is a decreasing sequence of open neighbourhoods of $\overline{Q}{}_G^0(\Zp)$, and $(V_r)_{r \ge 1}$ similarly of $Q^0_G(\Zp)$.

 \subsection{Hecke actions}

  For $r \ge 1$, the groups $U_r$ and $V_r$ have Iwahori decompositions, so we have an integrally normalised Hecke operator $\cT'_{\eta}$ on cohomology at these levels, defined in weight $\lambda$ using the normalisation factor $p^h$ where $h = \langle \eta, \lambda \rangle$. Similarly, we obtain an integrally normalised pushforward map
  \[ [\tau^r]_*: H^i(U_r, V_{\lambda, \cO}) \to H^i(V_r, V_{\lambda, \cO}), \]
  whose effect on the coefficient sheaves is given by $p^{hr} \rho_{\lambda}(\tau^{-r})$. Note that this map commutes with the actions of $\cT'_{\eta}$.

 \subsection{Pushforward from \texorpdfstring{$H$}{H}}

  \begin{definition}
   Let $c$ be the integer
   \[  \dim_{\RR}(Y_{\cG}) - \dim_{\RR}(Y_{\cH}) - \operatorname{rk}_{\RR}\left(\tfrac{Z_H}{Z_G \cap H}\right)  \ge 0,\]
   where $\operatorname{rk}_{\RR}$ denotes the split rank over $\RR$. (Note that this is simply $2\left( \dim_{\CC}(Y_{\cG}) - \dim_{\CC} Y_{\cH}\right)$, if $\cH$ and $\cG$ admit compatible Shimura data.)
  \end{definition}

  For any open compacts $U_G \subseteq G(\Zp)$ and $U_H \subset H(\Zp)$ with $U_H \subset H \cap U_G$, and any $H$-admissible weight $\lambda$, the maps of \cref{cor:branchingO} give us morphisms
  \[
   \iota_{\star}^{[\lambda]}: H^i\left(U_H, V^{H, \min}_{\mu, \cO} \right) \to H^{i+c}\left(U_G, V^G_{\lambda, \cO}\right).
  \]
  (As usual, we work with some fixed tame levels $K_{\cG}^p$ and $K_{\cH}^p$ away from $p$; we shall suppose $K_{\cH}^p = K_\cG^p \cap \cH(\Af^p)$.)

  \begin{remark}
   In the Shimura-variety setting, $c = 2d$ is always even, and the pushforward map for \'etale cohomology takes the form
   \[ H^i_{\et}\left(U_H, V^{H, \min}_{\mu, \cO}(j)\middle) \to H^{i+2d}_{\et}\middle(U, V^G_{\lambda, \cO}(j+d)\right) \]
   for any $i, j$, where $(j)$ denotes a cyclotomic twist.
  \end{remark}

  Now let $\lambda$ be a $Q_H^0$-admissible weight, and $\mu$ the corresponding weight for $H$. We suppose we are given a compatible family of classes (for some given $i \ge 0$),
  \[ z_H^{[\mu]} \in
   H^i_{\Iw}\left(Q_H^0, V_{\mu, \cO}^{H, \min} \right).
  \]

  \begin{remark}
   We remind the reader that the special case when $Q_H^0 = H$, $\mu = 0$, and $z_H^{[\mu]}$ is the identity class in $H^0(H(\Zp), \cO)$ is permitted (and not at all trivial).
  \end{remark}

  \begin{definition}
  \label{def:xi}
   We let $z_{G, r}^{[\lambda]} \in H^{i+c}\left(U_r, V_{\lambda, \cO}^G\right)$ be the following class: it is the image of $z_H^{[\mu]}$ under the composite map
   \begin{align*}
    H^i_{\Iw}\left(Q_H^0, V_{\mu, \cO}^{H, \min} \right)
    &\xrightarrow{\pr^\star} H^i_{\Iw}\left(Q_H^0 \cap u U_r u^{-1},
     V_{\mu, \cO}^{H, \min} \right) \\
    &\xrightarrow{\iota_{\star}^{[\lambda]}} H^{i+c}\left(u U_r u^{-1},V_{\lambda, \cO}^G\right) \\
    &\xrightarrow{[u]_*} H^{i+c}\left(U_r, V_{\lambda, \cO}^G\right).
   \end{align*}
   We set
   \[ \xi_{G, r}^{[\lambda]} = [\tau^r]_*(z_{G, r}^{[\lambda]}) \in H^{i+c}\left(V_r, V_{\lambda, \cO}^G\right).\]
   (Here $\pr^\star$ is the natural pullback map, and $[\tau]_*$ the integrally normalised pushforward as above.)
  \end{definition}

  \begin{theorem}
   If $\pr_{V_{r+1}, V_r, \star}$ denotes pushforward along the map $Y_{\cG}(V_{r+1}) \to Y_{\cG}(V_r)$, then we have
   \[ \pr_{V_{r+1}, V_r, \star}\left(\xi^{[\lambda]}_{G, r+1} \right)= \cT_\eta' \cdot \xi^{[\lambda]}_{G, r},\]
   for all $r \ge 1$.
  \end{theorem}

  \begin{proof} This is one of the main technical result of \cite{loeffler-spherical}, Proposition 4.5.2 of \emph{op.cit.}. (Note that the notations $V_n, \tau$ have the same meaning in \emph{op.cit.} as here, while the the $\xi_{G, r}$ of \emph{op.cit.} is our $\xi_{G, r}^{[\lambda]}$, and the Hecke operator $\cT$ is our $\cT'_{\eta}$.)\end{proof}

  \begin{definition}
   We define
   \[
    \xi_{G, \infty}^{[\lambda]} \coloneqq
    \left( (\cT_\eta')^{-r} \eQG \xi_{G, r}^{[\lambda]} \right)_{r \ge 1}
    \in \eQG H^{(i+c)}_{\Iw}\left(V_\infty, V^G_{\lambda, \cO}\right),
   \]
   and we let
   \[
    \iota_{\infty}^{[\lambda]} : H^i_{\Iw}\left(Q_H^0, V_{\mu, \cO}^{H, \min}\right) \longrightarrow
    \eQG H^{(i+c)}_{\Iw}\left(Q_G^0, V^G_{\lambda, \cO}\right)
   \]
   to be the map sending $z_H^{[\mu]}$ to $\xi_{G, \infty}^{[\lambda]}$.
  \end{definition}

  Note that $\lambda$ (and hence $\mu$) are fixed in the above construction. The goal of this paper is to compare the classes $\xi_{G, \infty}^{[\lambda]}$ for different values of $\lambda$.

\section{Moment maps}

 \subsection{Definitions}

  Let $r \ge 1$, and let $\lambda \in X^\bullet_+(S_G / S_G^0)$. (For the moment we do not assume $\lambda$ is $H$-admissible.)

  Since $f_\lambda^{\hw}$ is a highest-weight vector, and its weight is trivial on $L_G^0$, the mod $p^r$ reduction of this vector is invariant under the group $Q_G^0(\ZZ / p^r)$, which is the image of $V_r$ in $G(\ZZ/p^r)$. Thus the reduction of $f_\lambda^{\hw}$ defines a class
  \[ f_{\lambda, r}^{\hw} \in H^0(V_r, V_{\lambda, r}), \]
  and these have the property
  \begin{equation}
   \label{eq:adjunction}
   f_{\lambda, {r+1}}^{\hw} \bmod p^r = \pr_{V_{r+1}, V_r}^{\star}\left(f_{\lambda, r}^{\hw}\right)\quad \in H^0(V_{r+1}, V_{\lambda, r}).
  \end{equation}

  \begin{definition}
   The \emph{moment map} of weight $\lambda$ and level $r$ is the map
   \[ \mom^\lambda_G\ \ :\ \  H^c_{\Iw}\left(Q_G^0, \cO\right)
   \longrightarrow H^c_{\Iw}\left(Q_G^0, V_{\lambda, \cO}\right)
   \]
   defined as follows: if $\underline{x} = (x_s)_{s \ge 1} \in \varprojlim_s H^c\left(V_s, \cO\right)$, then $\mom_G^\lambda(\underline{x})$ is given by
   \[ \left((x_s \bmod p^s) \cup f_{\lambda, s}^{\hw}\right)_{s \ge 1} \in \varprojlim_s H^c\left(V_s, V_{\lambda, s}\right) \cong H^c_{\Iw}\left(Q_G^0, V_{\lambda, \cO}\right).\]
   which is well-defined by \eqref{eq:adjunction}.
  \end{definition}

  By construction this map commutes with the action of the Hecke operator $\cT'_{\eta}$, and hence with the anti-ordinary projector $e'_{Q_G}$. (It is \emph{not} compatible with the action of the abelian group $S_G(\Zp)$.) Moreover, we can make the same definitions with the lattice $V_{\lambda, \cO}$ replaced with $V_{\lambda, \cO}^{\min}$ (and the moment maps for the two lattices are compatible).

 \subsection{Twist-compatibility}

  We now bring the group $H$ back into the picture. Suppose $\lambda \in X^\bullet_+(S_G)$ is $Q_H^0$-admissible, and $\mu = \mu(\lambda)$ as above. We consider the following diagram of maps:
  \[
   \begin{tikzcd}
    H^i_{\Iw}\left(Q_H^0, \cO\right)
    \rar["\iota_{\infty}^{[0]}"] \dar["\mom^\mu_H"]&
   \eQG H^{(i+c)}_{\Iw}\left(Q_G^0, \cO\right)
    \dar["\mom_G^\lambda"]\\
    H^i_{\Iw}\left(Q_H^0, V_{\mu, \cO}^{H, \min}\right)
    \rar["\iota_{\infty}^{[\lambda]}"]&
    \eQG H^{(i+c)}_{\Iw}\left(Q_G^0, V^G_{\lambda, \cO}\right).
   \end{tikzcd}
  \]

  The following is the key technical result of the present paper:

  \begin{theorem}
   The above diagram is commutative.
  \end{theorem}

  \begin{remark}
   As in \cref{rmk:coset} above, this can be generalised so that instead of comparing trivial coefficients with coefficients in $V^G_{\lambda, \cO}$, we compare coefficients $V_{\nu, \cO}$ and $V_{\lambda + \nu, \cO}$, where $\nu$ is not necessarily trivial on $L^{\mathrm{der}}$. We leave the details to the interested reader.
  \end{remark}

  \begin{proof}
   From the definition of the moment maps, it suffices to prove the following claim: suppose we define $z_H^{[\mu]} = \mom_H^{\mu}(z_H^{[0]})$ and apply the constructions above in weight 0 and in weight $\lambda$. Then for each $r \ge 1$ we have the following equality of elements of $H^c(V_r, V_{\lambda, r})$:
   \[ \xi_{G, r}^{[\lambda]} \bmod p^r = (\xi_{G, r}^{[0]} \bmod p^r) \cup f_{\lambda, r}^{\hw}. \]

   Recall that $V_r = \tau^{-r} U_r \tau^r$. We thus have an isomorphism $\tau^r: Y_G(U_r) \to Y_G(V_r)$, and our (integrally normalised) pushforward map $H^c(Y_G(U_r), V_{\lambda, r}) \to H^c(Y_G(V_r), V_{\lambda, r})$ can be written more explicitly as the composite
   \[ H^c(Y_G(U_r), V_{\lambda, r}) \longrightarrow H^c(Y_G(U_r), (\tau^r)^* V_{\lambda, r}) \xrightarrow[\cong]{\ (\tau^r)_\star\ } H^c(Y_G(V_r), V_{\lambda, r}), \]
   where the first map is given by the morphism of sheaves $V_{\lambda, r} \to (\tau^r)^* V_{\lambda, r}$ corresponding to the endomorphism $A = p^{hr} \rho_{\lambda}(\tau^{-r})$ of $V_{\lambda, \cO}$.

   Since $V_{\lambda, \cO}$ is an admissible lattice, it is the direct sum of its intersections with the weight spaces of $V_{\lambda}$ for the action of $S$. On the highest-weight space, $A$ acts as the identity, and on all other weight spaces, it acts as multiplication by $p^{nr}$ for some $n \ge 1$. Hence, modulo $p^r$, the map $A$ factors through projection to the highest weight space.

   Since $A$ acts trivially on the highest-weight space, we have $A(f^{\hw}_{\lambda, r}) = (\tau^r)^{\star}(f^{\hw}_{\lambda, r})$. So it suffices to prove that
   \[ z_{G, r}^{[\lambda]} \bmod p^r = \Big( (z_{G, r}^{[0]} \bmod p^r) \cup f_{\lambda, r}^{\hw}\Big) \pmod{\ker A}, \]
   as elements of $H^{i+c}(U_r, V_{\lambda, r})$. Via adjunction between pushforward and pullback, we are reduced to showing that the classes $u^{-1} f_{\lambda, r}^H$ and $f_{\lambda, r}^{\hw}$, viewed as vectors in $V_{\lambda, r}$ invariant under $u^{-1} H u \cap U_r$, agree modulo $\ker(A)$; and this is precisely the statement of \cref{cor:integralrelation}.
  \end{proof}

  Applying the maps in the diagram to the class $z_H^{[0]}$, and defining $z_H^{[\mu]}$ using the moment maps as above, we have proved the following:

  \begin{theorem}
   \label{mainthm}
   Let $\xi_{G, \infty} \coloneqq \xi_{G, \infty}^{[0]} \in \eQG H^{(i+c)}_{\Iw}\left(Q_G^0, \cO\right)$. Then, for all $Q_H^0$-admissible $\lambda$ and all $r \ge 1$, we have
   \[ \mom_{G, V_r}^{\lambda}\left(\xi_{G, \infty} \right) = (\cT_\eta')^{-r} e'_{Q_G} \cdot \xi_{G, r}^{\lambda}.\]
  \end{theorem}

  (We shall give a formula in the next section for the moments of $\xi_{G, \infty}$ at parahoric level, which may be more convenient in applications than the level groups $V_r$ used above.)

  \begin{remark}
   Thus the single Iwasawa-cohomology class $\xi_{G, \infty}$ ``knows'' the classes $\xi_{G, r}^{[\lambda]}$ (or at least their projections to the $Q_G$-anti-ordinary part) for all $r$ and all $\lambda$. As $\xi_{G, \infty}$ lies in the module $\eQG H^{(i+c)}_{\Iw}\left(Q_G^0, \cO\right)$, which is finitely-generated over $\cO[[\fS]]$ by \cref{thm:rockwood}, this forces very strong compatibilities between the $\xi_{G, r}^{[\lambda]}$ for varying $\lambda$, which is the crucial input needed to prove explicit reciprocity laws.
  \end{remark}

  For the next corollary, we shall suppose that $H^*(-)$ denotes \'etale cohomology of the canonical model, and moreover that $G$ and $H$ are of PEL type, so we have a formalism of \emph{motivic cohomology with coefficients} (see e.g.~\S 3.1 of \cite{leiloefflerzerbes18}). In this setting, we shall say that a class in $H^*_{\et}(Y_G(U), V_{\lambda, \cO})$, is \emph{motivic} if, after extending scalars to the fraction field $K$ of $\cO$, it lies in the image of the natural map from motivic cohomology to \'etale cohomology (and similarly for $Y_H$ in place of $Y_G$).

  \begin{corollary}
   Suppose that for all $r \ge 1$ and all $Q_H^0$-admissible $\lambda$, the class $z_{H, r}^{[\mu(\lambda)]}$  is motivic. Then $\mom_{G, V_r}^{\lambda}\left(\xi_{G, \infty} \right)$ is motivic for all such $r$ and $\lambda$.
  \end{corollary}

   Note that the motivicity assumption on the classes $z_{H, r}^{[\mu(\lambda)]}$ is automatic in the ``algebraic cycle'' setting, i.e.~when $Q_H^0 = H$ (which implies $\mu(\lambda) = 0$ for all $\lambda$) and $z_{H, r}^{[0]}$ is the identity class in $H^0$, since the identity class in \'etale cohomology is vacuously motivic (it is the image of the identity class in motivic cohomology). A more subtle example where the motivicity condition also holds is when $z_{H, r}^{[\mu(\lambda)]}$ is the $\GL_2$ Eisenstein class, by the main result of \cite{Kings-Eisenstein}.

  \begin{remark}\label{rmk:knowncases}
   Special cases of this computation appear in many of our previous works, such as:
   \begin{itemize}

    \item Theorem 6.3.4 of \cite{KLZ17} is the case $\cG = \GL_2 \times \GL_2$, $\cH = \GL_2$, with $Q_H$ and $Q_G$ taken to be Borel subgroups.

    \item The case $\cG = \operatorname{Res}_{F / \QQ} \GL_2$ and $\cH= \GL_2$, with $F$ a quadratic field, is considered in \cite[Theorem 8.2.3]{leiloefflerzerbes18} (for $F$ real quadratic, using \'etale cohomology) and \cite[Proposition 5.6]{loefflerwilliams18} (for $F$ imaginary quadratic, using Betti cohomology). Again, we take $Q_H$ and $Q_G$ here to be Borel subgroups.

    \item Theorem 9.6.4 of \cite{LSZ21} is an analogous result for $\cG = \GSp_4$ and $\cH = \GL_2\times_{\GL_1} \GL_2$, taking $Q_G$ to be the Siegel parabolic subgroup. In this case we did not assume $L_G / L_G^0$ to be commutative.

    \item Theorem 11.2.1 of \cite{LSZ-unitary} is the case $G = \operatorname{GU}(2, 1)$, $H = \GL_2 \times_{\GL_1} \operatorname{Res}_{E/\QQ} \GL_1$, for $E$ an imaginary quadratic field.
   \end{itemize}

   In some of these works there are various correction terms appearing (such as the factor $(-2)^{-q}$ which appears in several places in \cite{LSZ21}) which are, in effect, a consequence of using the ``wrong'' normalisations for the branching maps. In the present work, we have avoided these correction terms, by normalising the branching maps for algebraic representations using the same element $u$ used to define the pushforward maps $\iota_{\infty}^{[0]}$.

   We briefly mention some related works by other authors in the literature. A result roughly equivalent to the above theorem in the case of Gross--Kudla--Schoen diagonal cycles (associated to $G = (\GL_2)^3$, $H = \GL_2$ embedded diagonally) has recently been proved in \cite{DRfamilies} and independently in \cite{BSV1}. The methods of these works (in particular the former) appear to be closely related to the present work. The case of $G = \GL_2$ and $H$ a non-split torus (the ``Heegner point'' setting) has also been explored, initially by Castella \cite{castella13} using very different methods based on $p$-adic $L$-functions, and in a more general setting by Disegni \cite{disegni-univGZ} using methods rather closer to those of this paper.
  \end{remark}

\section{Interpolating properties}

 \subsection{Parahoric-level classes}

  It will be convenient to work with a slightly different family of subgroups. Let $J_G = \{ g \in G(\Zp): g \bmod p \in Q_G\}$ and $\bar{J}_G = \{ g \in G(\Zp): g \bmod p \in \bar{Q}_G\}$. On the other hand, let $J_H^0 = \{ h: h \bmod p \in Q_H^0\}$. Then the quotient $J_H^0 \backslash G(\Zp) / \bar{J}_G$ is identified with the $Q_H^0(\FF_p)$-orbits on $\cF(\FF_p)$, and the double coset of $u$, for $u$ as in \S 3.1 above, represents the unique open orbit. Thus we have a pushforward map
  \[ u_\star \circ \iota_\star: H^i(J_H^0, V^{H, \min}_{\mu, \cO}) \to H^{i+c}(\bar{J}_G,  V^G_{\lambda, \cO}),
  \]
  and this map is independent of the choice of $u$ in its orbit (except via the normalisation of the branching map).

  \begin{proposition}
   The map $\cB: H^*(J_G, V_{\lambda, \cO}) \to H^*(\bar{J}_G, V_{\lambda, \cO})$ given by pullback to level $\bar{J}_G \cap J_G$ composed with the trace map, commutes with the action of the operators $\cT_{\eta}'$, and induces an isomorphism on the ordinary part for these operators.
  \end{proposition}

  \begin{proof}
   Let $\mathcal{A}$ be the map $H^*(\bar{J}_G, V_{\lambda, \cO}) \to H^*(J_G, V_{\lambda, \cO})$, given by the double coset $[J_G \eta(p)^{-1} \bar{J}_G]$ (normalised as in \cref{def:integralHecke}). Then one checks that the composite of $\mathcal{A}$ and $\cB$, either way around, is $\cT'_{\eta}$; and the result follows.
  \end{proof}

  For $\lambda$ an $H$-admissible weight, and $\mu$ the corresponding weight of $H$, we have a class
  \[ \xi_{J_G}^{[\lambda]} = \cB^{-1}\left(\eQ \cdot \left(u_\star \circ \iota_\star\right)(z_{J_H^0}^{[\mu]})\right) \in H^{i+c}(J_G, V_{\lambda, \cO}),\]
  where $z_{J_H^0}^{[\mu]}$ is the image of $z_H^{[\mu]}$ at level $J_H^0$, pulled back to $J_H^0 \cap u \bar{J}_G u^{-1}$.

  \begin{proposition}
   For each $Q_H^0$-admissible $\lambda$, we have
   \[ \xi_{J_G}^{[\lambda]} = \mom_{G, J_G}^{[\lambda]}\left(\xi_{G, \infty}\right).\]
  \end{proposition}

  \begin{proof}
   This follows by an argument using Cartesian diagrams which is very similar to the proof of the norm-compatibility relations; we leave the details to the reader.
  \end{proof}

  There is an analogous map $\cB^*: H^*(\bar{J}_G, V_{\lambda, \cO}^\vee) \to H^*(J_G, V_{\lambda, \cO}^\vee)$ defined by interchanging the roles of $Q_G$ and $\bar{Q}_G$, and one checks that $\cB^*$ is the transpose of $\cB$ with respect to Poincar\'e duality. This commutes with the ``dominant'' Hecke operator $\cT_\eta$ defined by the double coset of $\eta(p)$ (the transpose of $\cT'_{\eta}$). So, if we are given some $v \in H^*_c(J_G, V_{\lambda, \cO}^\vee)$ which is an ordinary eigenvector for $\cT_\eta$, in the cohomological degree complementary to $i + c$, we can define $\bar{v} = (\cB^*)^{-1}(v)$, which is a $\cT$-eigenvector at level $\bar{J}_G$; and the Poincar\'e duality pairing between $\xi^{[\lambda]}_{J_G}$ and $v$ is given by
  \[
   \left\langle \xi^{[\lambda]}_{J_G}, v\right\rangle_{J_G} =
   \left\langle \left(u_\star \circ \iota_\star\right)(z_{J_H}^{[\mu]}), \bar{v}\right\rangle_{\bar{J}_G} = \left\langle z_{J_H^0}^{[\mu]}, \iota^* u^* (\bar{v})\right\rangle_{J_H^0 \cap u\overline{J}_G u^{-1}}.
  \]

  \begin{remark}
   It is natural to expect that the latter expression -- an ``open-orbit cohomological period'' at parahoric level -- should be related to cohomological periods at prime-to-$p$ level, via an appropriate Euler factor. For examples of statements of this kind for specific groups $G$ and $H$, see e.g.~\cite[Theorem 5.7.6]{KLZ17} for $\GL_2 \subset \GL_2 \times \GL_2$, and the results for $\GSp_4$ proved in \cite{loeffler-zeta2} and recalled in the next section. It would be interesting to attempt to find a general recipe for this Euler factor; at least in the special case when $Q_H^0 = H$ and $Q_G$ is a Borel, this may be possible using the results of Sakellaridis, in particular \cite[Theorem 1.3.1]{sakellaridis13}.
  \end{remark}

\section{Applications to \texorpdfstring{$\GSp_4$}{GSp(4)}}

 We now consider applications of the above theory to Euler systems for $\GSp_4$.

 \subsection{A variant of the setting of \tps{\cite{LSZ21}}{[LSZ22]}}

  Let $0 \le q \le a, 0 \le r \le b$ be integers; and let $G = \GSp_4$. We choose a set of primes $S \ni p$, and auxiliary data $K_S, \underline{\phi}_S, W, c_1, c_2$, as in \S 8.4.6 of \cite{LSZ21}; these will remain fixed throughout the following discussion. For simplicity, we also assume that the function $\underline{\phi}_S$ has the form $\phi_{1, S} \otimes \phi_{2, S}$ where each $\phi_{i, S}$ vanishes at $(0, 0)$ and transforms under $\ZZ_S^\times$ by (the adelic character attached to) $\chi_i^{-1}$, where $\chi_i$ is a Dirichlet character unramified outside $S$. This allows us to define classes
  \[ {}_{c_1, c_2} z_{\et}^{[a,b,q,r]} \in H^4_{\et}\left(G(\Zp), \mathscr{D}^{a,b}_{\Zp}(3-q)\right), \]
  which are given by the same construction as \emph{loc.cit.} (with $M = 1$) but taking the test data at $p$ to be the spherical test data, so $\underline{\phi}_p = \operatorname{ch}(\Zp^4)$ and $\xi_p = \operatorname{ch}(G(\Zp))$.

  We now consider the group
  \[ e'_{B_G} H^4_{\et, \Iw}\Big(N_G(\Zp), \Zp(3)\Big), \]
  where $B_G$ is the upper-triangular Borel and $N_G$ its unipotent radical. This is a  finitely-generated module over the Iwasawa algebra of $T_G(\Zp)$, and we can identify this group with $(\Zp^\times)^3$ in such a way that the character $(a,b,q)$ corresponds to the highest weight of $\mathscr{D}^{a,b}(-q)$. Thus the fibre of the Iwasawa cohomology at $(a,b,q)$ maps canonically (via the edge map of \cref{eq:spseq}) to the ordinary part of the cohomology of $\mathscr{D}^{a,b}(3-q)$ at Iwahori level.

  Given an automorphic representation $\Pi$ of $G$ as in \S 10 of \emph{op.cit.}, of weight $(a,b)$, good ordinary at $p$, we can identify the $\Pi^*$-eigenspace in the ordinary cohomology at Iwahori level with the $\Pi^*$-eigenspace at prime-to-$p$ level, via a suitable Hecke operator (compare Note 17.1.6 of \cite{LZ20} for Siegel-parahoric level). We label the Hecke parameters of $\Pi$ as $(\alpha, \beta, \gamma, \delta)$ in increasing order of valuation, so $v_p(\alpha) = 0$ and $v_p(\beta) = a + 1$.

  \begin{theorem}
   There exists a class
   \[  {}_{c_1, c_2} z_{\et} \in H^4_{\et, \Iw}\Big(N_G(\Zp), \Zp(3)\Big) \mathop{\hat\otimes}\Zp[[\Zp^\times]]\]
   such that for all $a,b,q,r$ as above, and any cuspidal automorphic representation $\Pi$ of weight $(a, b)$ ordinary at $p$, the image of ${}_{c_1, c_2} z_{\et}$ under evaluation at $(a,b,q,r)$ and projection into $H^1(\QQ, W_{\Pi_f}^*(-q))$ is $(-2)^{-q} \cE \cdot \operatorname{pr}_{\Pi}\left({}_{c_1, c_2} z_{\et}^{[a,b,q,r]}\right)$, where $\cE$ is the Euler factor
   \begin{multline*}
    \left(1 - \tfrac{p^q}{\alpha}\right)\left(1 - \tfrac{\beta}{p^{1 + q}}\right)
       \left(1 - \tfrac{\gamma}{p^{1 + q}}\right)
       \left(1 - \tfrac{\delta}{p^{1 + q}}\right)\\
       \times
       \left(1 - \tfrac{p^{a +1 + r}\chi_2(p)}{\alpha}\right)
       \left(1 - \tfrac{p^{a +1 + r}\chi_2(p)}{\beta}\right)
       \left(1 - \tfrac{\gamma}{p^{a + 2 + r}\chi_2(p)}\right)
       \left(1 - \tfrac{\delta}{p^{a + 2 + r}\chi_2(p)}\right).
   \end{multline*}
  \end{theorem}

  \begin{remark}
   This refines \cite[Theorem 17.2.2]{LZ20}, in which $(a, q)$ were varying and $(b, r)$ were fixed (but only Siegel-ordinarity was assumed, rather than full Borel-ordinarity as here). In the above result, the parameter space is four-dimensional and all four parameters $(a,b,q,r)$ are varying independently.
  \end{remark}

  \begin{proof}We consider the groups $\wG = \GSp_4 \times \GL_1$, and $\wH = (\GL_2 \times_{\GL_1} \GL_2) \times \GL_1$. There is an obvious embedding $\wH \into \wG$, extending the embedding $H \into G$ used in \cite{LSZ21}. We can also view $\wH$ as a double cover of $\GL_2 \times \GL_2$ via $(h_1, h_2, z) \mapsto (h_1, zh_2)$. All of these maps are compatible with the standard Shimura data, if we take the Shimura cocharacter to project trivially into the $\GL_1$ factors of $\wG$ and $\wH$. The cohomology group $H^4_{\et}\Big(N_G(\Zp), \Zp(3)\Big) \mathop{\hat\otimes}\Zp[[\Zp^\times]]$ then has a natural interpretation as \'etale cohomology for $\wG$ with level $N_{\wG}(\Zp)$.

  We apply the machinery of \cite{loeffler-spherical} and the present paper to the groups $\wG$ and $\wH$, taking the input class $z_{\wH}^{[0]}$ to be the pullback of the Eisenstein--Iwasawa class for $\GL_2 \times \GL_2$; the associated mirabolic subgroup $Q_{\wH}^0$ is then given by the preimage of the usual mirabolic of $\GL_2 \times \GL_2$, namely
  \[ \left\{ \left(
   \left(\begin{matrix} x & \star \\ 0 & 1 \end{matrix} \right),
   \left(\begin{matrix} xy & \star \\ 0 & y^{-1} \end{matrix} \right),
   y
   \right) : x, y \in \mathbf{G}_m\right\}.\]
  This group has an open orbit on the Borel flag variety of $\wG$, with trivial stabiliser. We can now obtain a family of classes in the Iwahori level tower for $\wG$. Comparing these with their analogues at prime-to-$p$ level is done via an Euler-factor computation, for which we refer to \cite{loeffler-zeta2} (the second of the two formulae given in \S 6.3 of \emph{op.cit.}).\end{proof}

 \subsection{The setting of \tps{\cite{HJS20}}{[HJS20]}}

  We now let $\cG = \GSp_4 \times_{\GL_1} \GL_2$. As in \emph{op.cit.} we can construct classes
  \[ {}_e z^{[a,b,c,q]}_{\et} \in H^5_{\et}\left( G(\Zp), \mathscr{W}^{a,b,c,*}_{\Zp}(3-q)\right)\]
  where $\mathscr{W}^{a,b,c}$ is a certain representation of $G$ for $a,b,c \ge 0$, and $\max(a, c) \le q \le \min(a+b, a+c)$. As in the previous section, we are fixing some arbitrary test data away from $p$, including the smoothing parameter $e > 1$; and we are taking the test data at $p$ to be the spherical data. We have also relabelled the parameters slightly, so $q$ is $a + r$ in the notation of \emph{op.cit.}.

  \begin{theorem}
   There exists a class ${}_e z_{\Iw} \in e'_{B_G} H^5_{\et, \Iw}\left( N_G(\Zp), \Zp(3)\right)$ whose specialisation at an automorphic representation $\Pi \times \Sigma$ of weight $(a,b,c)$ twisted by $\det^q$, with $\Pi$ and $\Sigma$ good ordinary at $p$, is given by $\cE \cdot \pr_{\Pi \times \Sigma}\left({}_e z^{[a,b,c,q]}_{\et}\right)$, where $\cE$ is the Euler factor
   \[ \prod\left\{ \left(1 - \tfrac{\lambda}{p^{2 + q}}\right) : \lambda = \beta \fb,\gamma \fa, \gamma\fb, \delta \fa,  \delta\fb \right\} \cdot \prod \left\{ \left(1 - \tfrac{p^{1 + q}}{\lambda}\right) : \lambda = \alpha \fa, \alpha \fb, \beta \fa\right\} \]
  \end{theorem}

  Here $\fa$ and $\fb$ are the Hecke parameters of $\Sigma$ (with $v_p(\fa) = 0$, $v_p(\fb) = c + 1$).

  \begin{proof}
   We apply the machinery of the previous sections with $\cH= \GL_2 \times_{\GL_1} \GL_2$ embedded in the usual way, and an Eisenstein class on the first factor of $\cH$ alone. As before, we obtain a family of classes at Iwahori level; and these are related to their analogues at prime-to-$p$ level via an Euler-factor computation, which is again carried out in \cite{loeffler-zeta2} (section 7.1).
  \end{proof}

  This result is an important ingredient in the proof of new cases of the Bloch--Kato conjecture for $\GSp_4 \times \GL_2$ in \cite{LZ21-erl}.

 \subsection{The setting of \tps{\cite{LZvista}}{[LZ20b]}}

  Finally, we consider the triple product of cusp forms, as in \cite{LZvista}. This requires some care owing to the necessity of choosing a self-dual twist (compare \S 3.3 of op.cit.).

  Let $\cG = \GSp_4 \times \GL_2 \times \GL_2 \times \GL_1$, and let $\cH = \GL_2 \times_{\GL_1} \GL_2$ embedded via
  \[ (h_1, h_2) \mapsto (h_1 \boxplus h_2, h_1, h_2, \det h_1^{-1}).\]
  In this setting, we let $Q_H^0 = H$, and $Q_G$ the upper-triangular Borel of $G$. Then one computes that $H \cap u \bar{Q}_G u^{-1}$ is a copy of $\mathbf{G}_m$, embedded as $z \mapsto (z, z, z, z^{-2})$. We take $S_G^0$ to be the subgroup
  \[
   \left\{ \left(
     \left( \begin{smallmatrix} \star\\&\star\\&&z\\&&&z \end{smallmatrix}\right),
     \left( \begin{smallmatrix} \star \\ & z \end{smallmatrix} \right),
     \left( \begin{smallmatrix} \star \\ & z \end{smallmatrix} \right),
     z^{-2}
    \right) : z \in \mathbf{G}_m\right\}.
  \]
  With these choices, the group $\fS$ is once more isomorphic to $(\Zp^\times)^4$; and a $p$-adic modular form for $\cG$, of level $Q_G^0(\Zp)$, is identified with a triple of $p$-adic modular forms for the groups $\GSp_4$, $\GL_2$ and $\GL_2$ such that the product of their weight-characters is a square, together with a choice of square root of this character. Applying the machinery of the previous sections, combined with yet another zeta-integral computation from \cite{loeffler-zeta2} (\S 8 of \emph{op.cit.}), gives exactly Theorem 8.2.5 of \cite{LZvista}, showing that diagonal-cycle cohomology classes for $\GSp_4 \times \GL_2 \times \GL_2$ vary in $p$-adic families.

 \section*{Acknowledgements}

  We would like to thank both Arshay Sheth and the anonymous referee for their valuable feedback on the original draft of this paper.

\providecommand{\bysame}{\leavevmode\hbox to3em{\hrulefill}\thinspace}
\renewcommand{\MR}[1]{%
 MR \href{http://www.ams.org/mathscinet-getitem?mr=#1}{#1}.
}
\newcommand{\articlehref}[2]{\href{#1}{#2}}

\end{document}